\documentclass[a4paper,12pt]{amsart}
\usepackage{amssymb}
\usepackage{ifthen}
\usepackage[usenames]{color}
\usepackage{graphicx}
\nonstopmode \numberwithin{equation}{section}
\newtheorem{thm}{Theorem}[section]

\newtheorem{cor}{Corollary}[section]
\newtheorem{lem}{Lemma}[section]
\newtheorem{prop}{Proposition}[section]

\newtheorem{claim}{Claim}[section]
\newtheorem{subclaim}{Subclaim}[section]
\newtheorem{conj}[equation]{Conjecture}
\newtheorem{case}{Case}[section]
\newtheorem*{mysolution}{Solution}
\newtheorem{step}{Step}[section]
\theoremstyle{definition}
\newtheorem{defn}{Definition}[section]
\newtheorem{examp}{Example}[section]
\newtheorem{prob}[equation]{Problem}
\newtheorem{ques}[equation]{Question}
\newtheorem{rem}{Remark}[section]
\newcounter {own}
\def\theown {\thesection       .\arabic{own}}

\newenvironment{pf}[1][]{%
	\vskip 3mm
	\noindent
	\ifthenelse{\equal{#1}{}}%
	{{\slshape Proof. }}%
	{{\slshape #1.} }%
}%
{\qed\bigskip}

\newcounter{alphabet}
\newcounter{tmp}

\makeatletter

\newcommand{\Rref}[1]{\@ifundefined{r@#1}{}{\setcounter{tmp}{\ref{#1}}\Alph{tmp}}}

\def\be{\begin{equation}}
	\def\ee{\end{equation}}

\newcommand{\ben}{\begin{enumerate}}
	\newcommand{\een}{\end{enumerate}}

\newcommand{\blem}{\begin{lem}}
	\newcommand{\elem}{\end{lem}}
\newcommand{\bthm}{\begin{thm}}
	\newcommand{\ethm}{\end{thm}}
\newcommand{\bcor}{\begin{cor}}
	\newcommand{\ecor}{\end{cor}}
\newcommand{\beg}{\begin{examp}}
	\newcommand{\eeg}{\end{examp}}
\newcommand{\begs}{\begin{examples}}
	\newcommand{\eegs}{\end{examples}}
\newcommand{\bdefe}{\begin{defn}}
	\newcommand{\edefe}{\end{defn}}
\newcommand{\bprob}{\begin{prob}}
	\newcommand{\eprob}{\end{prob}}
\newcommand{\bques}{\begin{ques}}
	\newcommand{\eques}{\end{ques}}
\newcommand{\bei}{\begin{itemize}}
	\newcommand{\eei}{\end{itemize}}
\newcommand{\bcl}{\begin{claim}}
	\newcommand{\ecl}{\end{claim}}
\newcommand{\bscl}{\begin{subclaim}}
	\newcommand{\escl}{\end{subclaim}}
\newcommand{\bca}{\begin{case}}
	\newcommand{\eca}{\end{case}}
\newcommand{\bstep}{\begin{step}}
	\newcommand{\estep}{\end{step}}
\newcommand{\bsol}{\begin{mysolution}}
	\newcommand{\esol}{\end{mysolution}}
\newcommand{\bcon}{\begin{conj}}
	\newcommand{\econ}{\end{conj}}
\newcommand{\bcons}{\begin{conjs}}
	\newcommand{\econs}{\end{conjs}}
\newcommand{\bprop}{\begin{prop}}
	\newcommand{\eprop}{\end{prop}}
\newcommand{\br}{\begin{rem}}
	\newcommand{\er}{\end{rem}}
\newcommand{\brs}{\begin{rems}}
	\newcommand{\ers}{\end{rems}}
\newcommand{\bo}{\begin{obser}}
	\newcommand{\eo}{\end{obser}}
\newcommand{\bos}{\begin{obsers}}
	\newcommand{\eos}{\end{obsers}}
\newcommand{\bpf}{\begin{pf}}
	\newcommand{\epf}{\end{pf}}
\newcommand{\ba}{\begin{array}}
	\newcommand{\ea}{\end{array}}
\newcommand{\beq}{\begin{eqnarray}}
	\newcommand{\beqq}{\begin{eqnarray*}}
		\newcommand{\eeq}{\end{eqnarray}}
	\newcommand{\eeqq}{\end{eqnarray*}}

\begin{document}
	\title{Schwarz-Pick type lemma and Landau type theorem for $\alpha$-harmonic mappings}
	
	\author[Vibhuti Arora]{Vibhuti Arora}
	\address{Vibhuti Arora, Department of Mathematics, National Institute of Technology Calicut, Kerala 673 601, India.}
	\email{vibhutiarora1991@gmail.com, vibhuti@nitc.ac.in}

	\author{Jiaolong Chen}
	\address{Jiaolong Chen, Key Laboratory of Computing and Stochastic Mathematics (Ministry of Education), School of Mathematics and Statistics, Hunan Normal University, Changsha, Hunan 410081, P. R. China}
	\email{jiaolongchen@sina.com}
	
	\author[Shankey Kumar]{Shankey Kumar}
	\address{Shankey Kumar, Department of Mathematics, Indian Institute of Technology Madras, Chennai, 600036, India.}
	\email{shankeygarg93@gmail.com}
	
	\author{Qianyun Li${}^{~\mathbf{*}}$}
	\address{Qianyun Li, Key Laboratory of Computing and Stochastic Mathematics (Ministry of Education), School of Mathematics and Statistics, Hunan Normal University, Changsha, Hunan 410081, P. R. China}
	\email{liqianyun@hunnu.edu.cn}
	
	\keywords{Invariant Laplacian equation, Poisson-Szeg\"{o} integral, Schwarz-Pick type inequality, Landau type Theorem.\\
		$^{\mathbf{*}}$Corresponding author}
	

	\maketitle
	
	\makeatletter\def\thefootnote{\@arabic\c@footnote}\makeatother
	\begin{abstract}
		The aim of this paper is twofold.
		First, we obtain a Schwarz-Pick type lemma for the $\alpha$-harmonic mapping $u=P_{\alpha}[\phi]$,
		where $\phi\in L^{p}(\mathbb{S}^{n-1},\mathbb{R} )$ and $p\in[1,\infty]$.
		We get an explicit form of the sharp function $\mathbf{C}_{\alpha, q}(x)$ in the inequality
		$|\nabla u(x)| \leq \mathbf{C}_{\alpha, q}(x)\|\phi\|_{L^p(\mathbb{S}^{n-1}, \mathbb{ R} )}$.
		Second, we prove a Landau type theorem for $u=P_{\alpha}[\phi]$,
		where $\phi\in L^{\infty}(\mathbb{S}^{n-1},\mathbb{R}^{n})$.
		These  results generalize and extend the corresponding
		results due to Kalaj  (Complex Anal. Oper. Theory, 2024) and
		Khalfallah et al. (Mediterr. J. Math.,
		2021).
	\end{abstract}
	
	\section{Introduction  and statement of the main results}
	\subsection{Preliminaries}
	For $n\ge 2$, denote by $\mathbb{R}^{n}$   the usual real Euclidean space of dimension $n$.
	For $x=(x_{1},\ldots,x_{n})\in \mathbb{R}^{n}$, we denote the Euclidean length   of $x$ by $|x|$ and sometimes   identify each point $x$ with a column vector.
	For two column vectors $x,y\in \mathbb{R}^{n}$, the Euclidean inner product of $x$ and $y$ is given by
	$\langle x,y\rangle$.
	We use $\mathbb{B}^n( r)=\{x\in \mathbb{R}^{n}:|x|<r\}$ to denote the ball of radius $r>0$ with center $0$. In particular,
	we write  $\mathbb{B}^n=\mathbb{B}^n( 1)$.
	
	For a matrix $A=\big(a_{ij}\big)_{n\times n}\in \mathbb{R}^{n\times n}$, we will consider the matrix norm
	$\|A\|=\sup\{|A\xi|:\; \xi\in \mathbb{S}^{n-1}\}$ and the matrix function $l(A)=\inf\{|A\xi|:\; \xi\in \mathbb{S}^{n-1}\}$,
	where $\mathbb{S}^{n-1}=\{x\in \mathbb{R}^{ n}:|x|=1\}$.
	
	\subsubsection{The generalized hypergeometric series}
	
	For $a\in \mathbb{R}$ and $k\in\{0,1,2,\ldots\}$, we use $(a )_{k}$ to denote the factorial function
	with
	\begin{center}
		$(a )_{0}=1$ ($a\not=0$) and $(a)_{k}=a (a +1) \cdots(a +k-1)$ ($k\geq 1$).
	\end{center}
	If $a$ is neither zero nor a negative integer, then
	$$
	(a)_{k}=\frac{\Gamma(a+k) }{ \Gamma(a)} ,
	$$
	where $\Gamma$ is the Gamma function.
	Let $p$, $q=p-1$ be two positive integers and $s\in \mathbb{R}$. We define the generalized hypergeometric series by
	\begin{equation}\label{eq-1.1}
		{ }_{p} F_{q}\left(a_{1}, a_{2}, \ldots, a_{p} ; b_{1}, b_{2}, \ldots, b_{q} ; s\right)=\sum_{k=0}^{\infty} \frac{\left(a_{1}\right)_{k} \cdots\left(a_{p}\right)_{k}}{\left(b_{1}\right)_{k} \cdots\left(b_{q}\right)_{k}} \frac{s^{k}}{k !},
	\end{equation}
	where $a_{i}, b_{j}\in \mathbb{R}$ ($1\leq i\leq p$, $1\leq j\leq q$) and  $b_{j}$  is neither zero nor a negative integer.
	If $\sum_{j=1}^{q}b_{j}-\sum_{i=1}^{p}a_{i}>0$, then the series (\ref{eq-1.1}) converges absolutely in $[-1,1]$ (cf. \cite[Theorem 2.1.2]{an1999}).
	
	\subsubsection{Invariant  Laplacian  Equation}
	A mapping $u \in C^{2}\left(\mathbb{B}^{n}, \mathbb{R}^{n}\right)(n \geq 2)$ is called to be {\em $\alpha$-harmonic} if it satisfies the  (M\"{o}bius) invariant  Laplacian equation
	$$
	\Delta_{\alpha}u(x)=(1-|x|^{2})\left\{\left(1-|x|^{2}\right) \sum_{j=1}^{n} \frac{\partial^{2}u}{\partial x_{j}^{2}}-2\alpha \sum_{j=1}^{n} x_{j} \frac{\partial u}{\partial x_{j}}+\alpha(2-n-\alpha)u\right\}=0.
	$$
	In this paper, we call $\Delta_{\alpha}$ the (M\"{o}bius) invariant Laplacian, since
	\begin{equation*}
		\Delta_{\alpha}\left\{\left(\operatorname{det} D\psi (x)\right)^{\frac{n-2+ \alpha}{2 n}} u(\psi(x))\right\}=\left(\operatorname{det} D\psi (x)\right)^{\frac{n-2+\alpha}{2 n}}\left(\Delta_{\alpha} u\right)(\psi(x))
	\end{equation*}
	for every  $u \in C^{2}\left(\mathbb{B}^n,\mathbb{R}^n\right)$ and for every $\psi \in \mathcal{M }(\mathbb{B}^n )$ (see \cite[Proposition 3.2]{Liu09}).
	Here  $ \mathcal{M} (\mathbb{B}^n )$  denotes the group of those M\"{o}bius transformations that map $\mathbb{B}^{n}$ onto $\mathbb{B}^{n}$.

	When $\alpha=0$ (resp. $\alpha=2-n$), the mapping $u$ is called harmonic (resp. hyperbolic harmonic).
	The properties of these two classes
	of mappings have been investigated extensively by many authors  (cf. \cite{ABR92,Bur92,CH2013, JP99, sto2016}).
	
	The operator $\Delta_{\alpha}$ is closely related to polyharmonic mappings (see \cite{AH2014, Liu21}) and to solutions of the Weinstein equation (see \cite{leu}).
	For a discussion of the fundamental properties of the operator $\Delta_{\alpha}$ in the unit disk $\mathbb{D}$, the reader is referred to (\cite{chen23,chen15,k2024,KMM2021, Long2025,Ol14}).
	For the higher-dimensional setting, see (\cite{chenli2024A,chen24,Liu04, Liu09, ZHD24,zhou22}) and the references therein.
	In this paper,
	we focus our investigations on the high-dimensional case.
	In the rest of the paper, we always assume that  $n\geq3$.

	For $p\in(0,\infty]$ and $m,n\in \mathbb{Z}^{+}$, we use  $L^{p}(\mathbb{S}^{n-1},\mathbb{R}^{m})$ to denote the space of   those  measurable  mappings
	$f: \mathbb{S}^{n-1}\rightarrow\mathbb{ R}^{m}$ such that  $\|f\|_{L^p(\mathbb{S}^{n-1}, \mathbb{ R}^{m})}<\infty$, where
	$$\|f\|_{L^p(\mathbb{S}^{n-1}, \mathbb{ R}^{m})}
	=\begin{cases}
		\displaystyle \;\left(\int_{\mathbb{S}^{n-1}}|f(\zeta)|^{p} d \sigma(\zeta)\right)^{1 / p},  & \text{ if } p\in (0,\infty),\\
		\displaystyle \;\text{esssup}_{\zeta\in \mathbb{S}^{n-1}}  |f( \zeta)|,   \;\;\;\;& \text{ if } p=\infty.
	\end{cases}$$
	Here,  $d \sigma$  denotes the normalized surface measure on  $\mathbb{S}^{n-1}$   so that $\sigma(\mathbb{S}^{n-1})=1$.
	
	If $\phi\in L^{p}(\mathbb{S}^{n-1},\mathbb{R}^{m})$ with $p\in[1,\infty]$,  we define the invariant Poisson integral or   Poisson-Szeg\"{o} integral
	of $\phi$ in $\mathbb{B}^{n}$ by $P_{\alpha}[\phi]$, where
	\begin{equation*}\label{eq-01.4}
		P_{\alpha}[\phi](x)=\int_{\mathbb{S}^{n-1}} P_{\alpha}(x, \zeta) \phi(\zeta) d \sigma(\zeta),
	\end{equation*}
	
	\begin{equation}\label{eq-1.2}
		P_{\alpha}(x, \zeta)=C_{n, \alpha} \frac{\left(1-|x|^{2}\right)^{1- \alpha}}{|x-\zeta|^{n- \alpha}}
		\;\;\text{and}\;\;C_{n, \alpha}=\frac{\Gamma\left(\frac{n-\alpha}{2}\right) \Gamma(1-\frac{\alpha}{2})}{\Gamma\left(\frac{n}{2}\right) \Gamma(1- \alpha)}
	\end{equation}
	(cf. \cite[Section 1]{li24}).
	
	In \cite{Liu04}, Liu and Peng  investigated  the solvability of the
	Dirichlet problem associated with the (M\"{o}bius) invariant Laplacian
	\begin{equation}\label{eq-1.3}
		\left\{\begin{array}{ll}
			\Delta_{\alpha} u(x)=0, & \text{ if} \;x\in\mathbb{B}^{n}, \\
			u(\zeta)=\phi(\zeta),& \text{ if}\; \zeta\in\mathbb{S}^{n-1},
		\end{array}\right.
	\end{equation}
	where $\phi\in C ( \mathbb{S}^{n-1} ,\mathbb{ R}^{n})$.
	They showed that the Dirichlet problem \eqref{eq-1.3} has a solution if and only if $\alpha<1$ (see \cite[Theorem 2.4]{Liu04}).
	In this case, the solution is unique and
	$$
	u(x)= P_{\alpha}[\phi](x).	
	$$

	\subsection{Main results}	
	
	The classical Schwarz-Pick lemma \cite{alh} states that a holomorphic function  $f$  from  $\mathbb{D}$  into itself   satisfies the inequality  $$|f'(z)| \leq \frac{1-|f(z)|^{2}}{1-|z|^{2}}$$ for all  $z \in \mathbb{D}$.
	Later,  Colonna \cite{col} generalized the result to harmonic functions. He showed that
	if $f$ is a harmonic function from $\mathbb{D}$  into itself, then for all  $z \in \mathbb{D}$,
	$$|f_{z}(z)|+|f_{\overline{z}}(z)|\leq \frac{4}{\pi} \frac{1}{1-|z|^{2}}.$$
	
	In \cite[Lemma 3.2]{li24}, the second and the last authors of this paper,
	obtained a Schwarz-Pick  type lemma for $\alpha$-harmonic mappings at the point $x=0$.
	If $u=P_{\alpha}[\phi]$  in $\mathbb{B}^{n}$ and
	$\phi\in L^p(\mathbb{S}^{n-1},\mathbb{R}^{n})$,
	then the following sharp inequality holds:
	\begin{equation}\label{eq-1.4}
		\|D u(0) \| \leq (n-\alpha) C_{n,\alpha}\left(\frac{\Gamma\left(\frac{n}{2}\right) \Gamma\left(\frac{1+q}{2}\right)}{\sqrt{\pi} \Gamma\left(\frac{n+q}{2}\right)}\right)^{\frac{1}{q}}\|\phi\|_{L^p(\mathbb{S}^{n-1}, \mathbb{ R}^{n} )},
	\end{equation}
	where $\alpha<1$, $p\in[1,\infty]$ and $q$ is its conjugate.
	Here and hereafter, $Du(x)$ denotes the usual Jacobian matrix of $u$ at $x$
	$$
	Du(x)=\left(
	\begin{array}{cccc}
		\frac{\partial u_{1}(x)}{\partial x_{1}}    & \cdots & \frac{\partial u_{1}(x)}{\partial x_{n}} \\
		\vdots & \ddots  & \vdots \\
		\frac{\partial u_{n}(x)}{\partial x_{1}}  & \cdots & \frac{\partial u_{n}(x)}{\partial x_{n}}\\
	\end{array}
	\right) = \big(\nabla u_1(x) \cdots \nabla u_n(x)\big)^T
	\;\text{and}\;
	u(x)=\left(
	\begin{array}{cccc}
		u_{1}(x) \\
		\vdots    \\
		u_{n}(x) \\
	\end{array}
	\right),
	$$
	where $T$ is the transpose and the gradients $\nabla u_j$ are understood as column vectors.
	
	Recently, Kalaj \cite[Theorem 1.2]{k2024} established a Schwarz-Pick lemma for $\alpha$-harmonic functions in  $\mathbb{D}$.
	For some other generalizations of the Schwarz-Pick lemma
	we refer to the papers \cite{CH2013, kp2012, KMP2022, LMW25, Liu22, zhu19}.

	The first aim of this paper is to generalize  \cite[Theorem 1.2]{k2024} to the higher-dimensional case. Our results are as follows.

	\begin{thm}\label{thm-1.1}
		Let $\alpha<1$, $p\in(1,\infty]$  and $q$ be its conjugate.
		If $u=P_{\alpha}[\phi]$ in $\mathbb{B}^{n}$ with $\phi\in L^{p}(\mathbb{S}^{n-1},\mathbb{R})$, then
		for any $x \in \mathbb{B}^{n} $,
		\begin{align}\label{eq-1.5}
			|\nabla u(x)| \leq
			\frac{(n-\alpha)C_{n,\alpha}}{\left(1-|x|^{2}\right)^{\frac{n(q-1)+1}{q}}}\left(\sup_{l \in\mathbb{S}^{n-1}}I(\alpha,q,x,l)+J(\alpha,q,x )\right)^{\frac{1}{q}}\|\phi\|_{L^p( \mathbb{S}^{n-1}, \mathbb{ R} ) },
		\end{align}
		where $\sup_{l \in\mathbb{S}^{n-1}}I(\alpha,q,x,l)$ and $ J(\alpha,q,x )$ are the mappings from Lemma \ref{lem-2.4} and
		\eqref{eq-2.7}, respectively. The above inequality is sharp when $x=0$ or $\alpha=2-n$.
		
	\end{thm}
	
	\medskip

	\begin{thm}\label{thm-1.2}
		Let $\alpha<1$ and $u=P_{\alpha}[\phi]$ in $\mathbb{B}^{n}$ with $\phi\in L^{1}(\mathbb{S}^{n-1},\mathbb{R})$.
		
		$(1)$ If $2-n\leq\alpha<1$, then for any $x \in \mathbb{B}^{n} $,
		$$
		|\nabla u(x)|
		\leq C_{n,\alpha} \frac{ n-\alpha+(n+\alpha-2)|x| }{(1-|x|)^{n}} \|\phi\|_{L^{1}( \mathbb{S}^{n-1}, \mathbb{ R} ) };
		$$
		
		$(2)$ 	If $\alpha<2-n$,  then for any $x \in \mathbb{B}^{n} $,
		$$
		|\nabla u(x)|
		\leq  C_{n,\alpha}    \frac{ n-\alpha-(n+\alpha-2)|x| }{(1-|x|)^{ n}}\|\phi\|_{L^{1}( \mathbb{S}^{n-1}, \mathbb{ R} ) }.
		$$
		The above two inequalities are sharp at the point $x=0$ or $2-n\leq\alpha<1$.
	\end{thm}

	Next, we discuss an issue that is related to a classical result in geometric function theory:
	the theorem of Landau.
	Recall that Landau's theorem says that if $f$ is analytic in $\mathbb{D}$, with $f(0)=f'(0)-1=0$
	and $|f(z)|<M$ for some $M\geq1$, then $f$ is univalent in the disk $\mathbb{D}(\rho)=\{z\in \mathbb{C}:|z|<\rho\} $ with
	$$\rho=\frac{1}{M+\sqrt{M^{2}-1}}>0.$$ In addition, the range $f\big(\mathbb{D}(\rho)\big)$ contains a
	disk of radius $M\rho^{2}$ (cf. \cite{ch}).
	There have been a number of articles dealing with Landau's theorem for planar harmonic mappings  (cf. \cite{allu, cg, dorff, gri06,liu2009}).
	See also \cite{cz, lip, luo}  for similar discussions in this line.
	Recently,  Khalfallah et al. \cite[Theorem 7]{KMM2021} obtained   a Landau theorem for $\alpha$-harmonic mappings in   $\mathbb{D}$.
	As the second aim of this paper, we consider Landau's theorem for $\alpha$-harmonic mappings in $\mathbb{B}^{n}$.
	
	Before the statement  of our third main result, for convenience, we introduce
	the following notational conventions.
	
	For a mapping $u\in C^{1}(\mathbb{B}^{n},\mathbb{R}^{n})$, we use $J_{u}(x)$ to denote
	the Jacobian of $u$ at the point $x$, i.e., $$J_{u}(x)=\text{det}\big(Du(z)\big).$$
	
	For any $r\in[0,1)$ and $\alpha<1$, let
	\begin{align}\label{eq-1.6}
		&G (r)
		=2(1-\alpha) C_{n,\alpha}  \max_{0\leq |x|\leq r }\left\{{}_{2}F_{1}\left(\frac{\alpha}{2} , \frac{n+\alpha}{2}-1; \frac{n}{2} ;|x|^{2}\right)\right\}\frac{ 1}{(1-r^{2})} \\\nonumber
		& +(n-\alpha) C_{n,\alpha} \max_{0\leq|x|\leq r }\left\{{}_{2}F_{1}\left(\frac{\alpha}{2}-1 , \frac{n+\alpha}{2}-2; \frac{n}{2} ;|x|^{2}\right)\right\}\frac{ 1}{    (1-r^{2})^{2}}   \\
		& +(n-\alpha) C_{n,\alpha} \max_{0\leq|x|\leq r }\left\{ {}_{2}F_{1}\left(\frac{\alpha}{2}-1 , \frac{n+\alpha}{2}-2; \frac{n}{2} ;|x|^{2}\right)g(|x|) \right\}\frac{1}{(1-r^{2})^{3-\alpha }}\nonumber
	\end{align}
	and set
	\begin{align}\label{eq-1.7}
		g(r ) =\left\{\begin{array}{ll}
			\big( 2(1 +r^{2})^{\frac{n+2-\alpha}{2}}
			-(1-r )^{n+2-\alpha  }   -(1-r^{2})^{ 1-\alpha}\big)/r, & \text{ if} \;r\in(0,1), \\
			n+2-\alpha,& \text{ if}\; r=0.
		\end{array}\right.
	\end{align}
	
	
	\begin{thm}\label{thm-1.3}
		Let $\alpha<1$, $u=P_{\alpha}[\phi]$ in $\mathbb{B}^{n}$ with $\|\phi\|_{L^{\infty}(\mathbb{S}^{n-1},\mathbb{R}^{n})}\leq M$, where $M$ is a positive constant.
		If $ u(0)=J_{u}(0)-1=0$,
		then  $u $
		is univalent in $\mathbb{B}^{n}(r_{0})$
		and $u(\mathbb{B}^{n}(r_{0}))$
		contains a ball  $\mathbb{B}^{n}( R_{0})$, where
		$$R_{0}=\frac{M}{2} r^{2}_{0}G(r_{0})$$
		and $r_{0}$ is   the smallest
		positive root of the equation
		$$
		M^{n}rG(r)
		= \left(\frac{ \sqrt{\pi}  \Gamma \left(\frac{n+1}{2}\right)}{(n-\alpha)  C_{n,\alpha}  \Gamma\left(\frac{n}{2}\right)} \right)^{n-1}.
		$$
		
	\end{thm}
	
	The rest of this paper is organized as follows.
	In Sections  2 and 3, we present the proofs of Theorems 1.1 and 1.2, respectively.
	In  Section 4,   we give the proof of Theorem 1.3.
	
	\section{Proof of Theorem \ref{thm-1.1}}

	The aim of this section is to prove Theorem \ref{thm-1.1}.
	First, we introduce some necessary terminology.
	Then, we prove some auxiliary results.
	
	Let $p \in[1, \infty]$ and  $q$  be its conjugate. Assume that $ u=P_{\alpha}[\phi]$, where $\phi \in L^{p}\left(\mathbb{S}^{n-1}, \mathbb{R}\right)$.
	For any $x \in \mathbb{B}^{n} $  and  $l \in \mathbb{S}^{n-1}$, let $\mathbf{C}_{\alpha,q}(x)$  and  $\mathbf{C}_{\alpha,q}(x ; l)$  denote the optimal numbers such that
	\begin{equation}\label{eq-2.1}
		|\nabla u(x)| \leq\mathbf{C}_{\alpha, q}(x)\|\phi\|_{L^p(\mathbb{S}^{n-1}, \mathbb{ R} )}
		\;\;
		\mbox{and}
		\;\;
		|\langle\nabla u(x), l\rangle| \leq\mathbf{C}_{\alpha, q}(x ; l)\|\phi\|_{L^p(\mathbb{S}^{n-1}, \mathbb{ R} )}.
	\end{equation}
	Since  $|\nabla u(x)|=\sup _{l \in \mathbb{S}^{n-1}}|\langle\nabla u(x), l\rangle|$,   we obtain
	\begin{equation}\label{eq-2.2}
		\mathbf{C}_{\alpha, q}(x)=\sup _{l \in \mathbb{S}^{n-1}}\mathbf{C}_{\alpha, q}(x ; l).
	\end{equation}
	In \cite{Liu21B}, Liu showed that $\mathbf{C}_{0,  1}(x)=\mathbf{C}_{0, 1}\left(x ; n_{x}\right)$ for any $x\in \mathbb{B}^{n}$,
	where $n_{x}$ is the unit vector such that $x=|x|n_{x} $.
	Later, Khalfallah et al. \cite{KHM2022} proved that
	$$\mathbf{C}_{2-n,  q}(x)=\begin{cases}
		\displaystyle \; \mathbf{C}_{2-n,  q} (x ; n_{x} ), & \text{ if } q\in( \frac{n}{n-1} ,\infty),\\
		\displaystyle \; \mathbf{C}_{2-n,  q} (x ; t_{x} ),  \;\;\;\;& \text{ if } q\in(1,\frac{n}{n-1} ),
	\end{cases}$$
	$\mathbf{C}_{2-n,  \frac{n}{n-1}}(x)\equiv\mathbf{C}_{2-n,  \frac{n}{n-1}}\left(x ; l\right)$
	and
	$\mathbf{C}_{2-n,  1}(x)\equiv\mathbf{C}_{2-n,  1}\left(x ; l\right)$
	for any $l\in  \mathbb{S}^{n-1}$,
	where $t_{x} $ is any unit vector such that $\langle t_{x},n_{x}\rangle =0$.
	See   \cite{chen2023B, k2017,K1992,KM2010A, M2017,P2019}
	and references therein for more discussion along this line.

	\subsection{Several Lemmas}
	\begin{lem}\label{lem-2.1}
		For any $\alpha<1$, $q \in[1, \infty)$, $l \in \mathbb{S}^{n-1}$ and  $x \in \mathbb{B}^{n}$, we have
		\begin{align*}
			&\mathbf{C}_{\alpha, q}(x ; l)\\
			&=\frac{(n-\alpha)C_{n,\alpha}}{\left(1-|x|^{2}\right)^{\frac{n(q-1)+1}{q}}} \left(\int_{\mathbb{S}^{n-1}}|\eta-x|^{( n-\alpha)q-2n+2}\bigg|\left\langle\eta+\frac{2-\alpha-n}{n-\alpha}x, l\right\rangle\bigg|^{q} d \sigma(\eta)\right)^{\frac{1}{q}}.
		\end{align*}
	\end{lem}
	\begin{proof}
		The proof of this result is similar to that of \cite[Lemma 2.1]{chen2023B}, where $P_{\alpha}$ is used
		instead of $\mathcal{P}_{\mathbb{B}^{n}}$,  and we omit it.
	\end{proof}

	For any $q \in[1, \infty)$, $l \in \mathbb{S}^{n-1}$ and $x \in \mathbb{B}^{n}$, set
	\begin{equation}\label{eq-2.3}
		C_{\alpha, q}(x ; l)=\int_{\mathbb{S}^{n-1}}|\eta-x|^{( n-\alpha)q-2n+2}\left|\left\langle\eta+\frac{2-\alpha-n}{n-\alpha}x, l\right\rangle\right|^{q} d \sigma(\eta).
	\end{equation}	
	It follows from Lemma  \ref{lem-2.1} that
	\begin{equation}\label{eq-2.4}
		\mathbf{C}_{\alpha, q}(x ; l) =\frac{(n-\alpha)C_{n,\alpha}}{\left(1-|x|^{2}\right)^{\frac{n(q-1)+1}{q}}}C^{\frac{1}{q}}_{\alpha, q}(x ; l).
	\end{equation}	
	
	In the following, we will estimate the quantity $C_{\alpha, q}(x ; l)$.
	
	\begin{lem}\label{lem-2.2}
		For any $\alpha<1$, $ q \in[1, \infty)$ and  $x \in \mathbb{B}^{n}$, we have
		\begin{equation}\label{eq-2.5}
			C_{\alpha, q}(x ; l)\leq I(\alpha,q,x,l)+  J(\alpha,q,x ),
		\end{equation}
		where
		\begin{equation}\label{eq-2.6}
			I(\alpha,q,x,l)=\int_{\mathbb{S}^{n-1}}|\eta-x|^{( n-\alpha)q-2n+2}|\langle\eta, l\rangle|^{q} d \sigma(\eta)
		\end{equation}
		and
		\begin{align}\label{eq-2.7}
			J(\alpha,q,x )
			=& q |x|\frac{\left |2-\alpha-n\right| }{n-\alpha}\left( 1+ \frac{\left| 2-\alpha-n\right|}{n-\alpha}   |x| \right)^{q-1}\notag\\
			&\times\; _{2}F_{1}\left(n-1-\frac{q(n-\alpha)}{2}, \frac{n-q(n-\alpha)}{2},\frac{n}{2},|x|^{2}\right).
		\end{align}

	\end{lem}
	\begin{proof}
		For any $t\in \mathbb{R}$ and $ q \in[1, \infty)$,	let $f(t)=t^{q}$.
		By the Lagrange mean value theorem, we have
		$$
		|f(t+s)|\leq |f(t)|+|s|\max_{\theta\in[0,1]}\big|f' (  t+ \theta s )\big|
		$$
		for any $s\in \mathbb{R}$.
		Therefore,
		\begin{align}\label{eq-2.8}
			&\Bigg|\left \langle   \eta, l\right \rangle + \left \langle  \frac{2-\alpha-n}{n-\alpha}x, l\right \rangle\Bigg|^{q}\\\nonumber
			&\leq \left|\left \langle   \eta, l\right \rangle\right|^q
			+ q |  x |\frac{|2-\alpha-n|}{n-\alpha} \left( 1+ \frac{\left|2-\alpha-n\right|}{n-\alpha}  |x| \right)^{q-1}.
		\end{align}
		Notice that \cite[Lemma 2.1]{Liu04} ensures
		$$
		\int_{\mathbb{S}^{n-1}}|\eta-x|^{( n-\alpha)q-2n+2} d \sigma(\eta) =
		{}_{2}F_{1}\left(n-1-\frac{q(n-\alpha)}{2}, \frac{n-q(n-\alpha)}{2},\frac{n}{2},|x|^{2}\right).
		$$
		Combining this with \eqref{eq-2.3} and \eqref{eq-2.8},
		we see that
		\eqref{eq-2.5} holds true.
		The proof of the lemma is complete.
	\end{proof}
	
	Next, we will calculate the quantity  $\sup_{l\in\mathbb{S}^{n-1}}I(\alpha, q,x , l )$. For this,  we need to consider the following integral
	\begin{equation}\label{eq-2.9}
		\mathcal{I}(A,B; a, b ; \beta)=\int_{-\pi}^{\pi}(A-B\cos\theta)^{a}\left|\cos(\theta-\beta)\right|^{b}d\theta,
	\end{equation}
	where $\beta\in[0,\pi]$ and $ A$, $B$, $a$, $b$ are real numbers such that $0\leq B< A$ and $b>0$.
	It is straightforward to see from the form of $\mathcal{I}(A,B;a, b ; \beta)$  that the mapping
	$\beta\mapsto\mathcal{I}(A,B;a, b ; \beta)$ is $\pi$-periodic and an even function.
	Further, by the proof of \cite[Lemma 5]{KHM2022}, we obtain the following properties of $\mathcal{I}(A,B;a, b ; \beta)$.
	
	\medskip
\noindent\textbf{Lemma A.} Let $\beta\in[0,\pi]$ and $ A$, $B$, $a$, $b\in \mathbb{R}$ with $0\leq B< A$ and $b>0$.
		
		$(1)$ If  $a=0$  or  $a=1$, then  $\beta \mapsto \mathcal{I}(A,B;a, b ; \beta)$  is a constant on $\left[0,  \pi \right]$;
		
		$(2)$  If $a \in(0,1)$, then $\beta \mapsto \mathcal{I}(A,B;a, b ; \beta)$ is increasing on $\left[0, \frac{\pi}{2}\right]$;
		
		$(3)$  If  $a>1$ or $a<0$, then  $\beta \mapsto \mathcal{I}(A,B;a, b ; \beta)$  is decreasing on  $\left[0, \frac{\pi}{2}\right]$.

	For any  $\alpha<1$,  $\beta\in[0,\pi]$, $q \in[1, \infty)$ and  $r,s\in[0,1)$, set
	\begin{equation}\label{eq-2.10}
		\mathcal{J} (\alpha,\beta,q,r, s)=
		\mathcal{I}\left(1+s^{2},2  rs;\frac{\left( n-\alpha\right)q}{2}-n+1, q ; \beta\right).
	\end{equation}

	Obviously,  Lemma A implies the following corollary, which is useful for calculating  $\sup_{l\in\mathbb{S}^{n-1}}I(\alpha, q,x , l )$.
	\begin{cor}\label{cor-2.1}
		Let $\alpha<1$,  $\beta\in[0,\pi]$, $q \in[1, \infty)$ and $r,s\in[0,1)$.
		
		$(1)$  If  $q=\frac{2n-2}{n-\alpha}$  or  $q=\frac{2n}{n-\alpha}$, then $ \beta \mapsto \mathcal{J} (\alpha,\beta,q,r, s)$ is a constant on  $\left[0,  \pi \right]$;
		
		$(2)$If  $\frac{2n-2}{n-\alpha}<q<\frac{2n}{n-\alpha} $, then $ \beta \mapsto \mathcal{J} (\alpha,\beta,q,r, s)$ is  increasing on  $\left[0, \frac{\pi}{2}\right]$  and
		$$
		\max _{\beta \in\left[0,  \pi \right]} \mathcal{J} (\alpha,\beta,q,r, s)= \mathcal{J} \left(\alpha,\frac{\pi}{2},q,r, s\right);
		$$

		$ (3)$ If $q>\frac{2n}{n-\alpha}$  or  $1\leq q<\frac{2n-2}{n-\alpha}$, then $ \beta \mapsto \mathcal{J} (\alpha,\beta,q,r, s)$ is  decreasing on  $\left[0, \frac{\pi}{2}\right]$  and
		$$
		\max _{\beta \in\left[0,  \pi \right]} \mathcal{J} (\alpha,\beta,q,r, s)= \mathcal{J} (\alpha,0,q,r, s).
		$$
	\end{cor}

	For any $i\in\{1,2,\ldots,n\}$  and $\beta\in[0,\pi]$, we use $e_{i}\in \mathbb{S}^{n-1}$ to denote the unit vector with $i$-th coordinate 1 and
	set $l_{\beta}=\cos \beta e_{1}+\sin \beta e_{2}$.
	Then we have the following result.
	\begin{lem}\label{lem-2.3}
		For any $\alpha<1$, $q \in[1, \infty)$,  $x \in \mathbb{B}^{n}$ and $l\in \mathbb{S}^{n-1}$,  we have
		$$
		I(\alpha, q,x , l )=I\left(\alpha,q,|x| e_{1}, l_{\beta}\right)
		=\frac{n-2}{2 \pi} \int_{0}^{1}\left(1-r^{2}\right)^{\frac{n-4}{2}} r^{q+1} \mathcal{J} (\alpha,\beta,q,r, |x| ) d r,
		$$
		where $\beta\in[0,\pi]$ and $\langle x,l \rangle=|x|\cos \beta$.
	\end{lem}
	
	\begin{proof}
		The proof of this result is similar to that of \cite[Lemma 4]{KHM2022},   and we omit it.
	\end{proof}

	By Lemma \ref{lem-2.3} and Corollary \ref{cor-2.1}, we easily obtain the following corollary.	
	
	\begin{cor}\label{cor-2.2}
		Let  $ \alpha<1$, $q \in[1, \infty)$  and  $x \in \mathbb{B}^{n} $.
		\begin{itemize}
			\item[(1)] If $q=\frac{2n-2}{n-\alpha}$ or  $q=\frac{2n}{n-\alpha}$,
			then  $l \mapsto I\left(\alpha,q,x, l \right)$ is a constant;
			\item[(2)] If  $\frac{2n-2}{n-\alpha}<q<\frac{2n}{n-\alpha} $, then
			$$
			\sup_{l \in\mathbb{S}^{n-1}}I(\alpha, q,x , l )=\max _{\beta \in[0, \pi / 2]} I\left(\alpha,q,|x| e_{1}, l_{\beta}\right)=I\left(\alpha,q,|x| e_{1}, l_{\frac{\pi}{2}}\right);
			$$
			
			\item[(3)] If $1\leq q<\frac{2n-2}{n-\alpha}$ and $q>\frac{2n}{n-\alpha}$, then
			$$
			\sup_{l \in\mathbb{S}^{n-1}}I(\alpha, q,x , l )=\max _{\beta \in[0, \pi / 2]} I\left(\alpha,q,|x| e_{1}, l_{\beta}\right)=I\left(\alpha,q,|x| e_{1}, l_{0}\right).
			$$

		\end{itemize}
	\end{cor}
	
	Now, we will give a very explicit representation of the quantity $\sup_{l\in\mathbb{S}^{n-1}}I(\alpha, q,x , l )$.
	
	\begin{lem}\label{lem-2.4}
		Let  $\alpha<1$, $ q \in[1, \infty)$ and  $x\in\mathbb{B}^{n} $.
		
		$(1)$  If $q=\frac{2n-2}{n-\alpha}$, then
		$$
		\sup_{l \in\mathbb{S}^{n-1}}I(\alpha, q,x , l )=
		\frac{\Gamma(\frac{n}{2}) \Gamma\left(\frac{3 n-\alpha-2}{2 n-2\alpha }\right)}{\sqrt{\pi} \Gamma \left(\frac{n^{2}+2n -\alpha n -2}{2 n-2\alpha }\right) }.
		$$

		$(2)$ If  $q=\frac{2n}{n-\alpha}$, then
		$$\sup_{l \in\mathbb{S}^{n-1}}I(\alpha, q,x , l )= \frac{\Gamma(\frac{n}{2})\Gamma(\frac{3n-\alpha}{2n-2\alpha})}{\sqrt{\pi}\Gamma(\frac{n^{2}+2n -\alpha n }{2n-2\alpha})}(1+|x|^{2}).
		$$

		$(3)$ If 	$\frac{2n-2}{n-\alpha}<q<\frac{2n}{n-\alpha} $, then
		\begin{align*}
			&\sup_{l \in\mathbb{S}^{n-1}}I(\alpha, q,x , l )= I\left(\alpha,q,x,  t_{x}\right)\\
			& =\frac{\Gamma\left(\frac{n}{2}\right)\Gamma\left(\frac{q+1}{2}\right) }{\sqrt{\pi}\Gamma\left(\frac{q+n}{2}\right) }
			{ }_{2} F_{1}\left(n-1-\frac{\left(n-\alpha\right)q}{2}, \frac{n -(n+1-\alpha )q}{2}; \frac{q+n}{2}; |x|^{2}\right)
		\end{align*}
		and
		\begin{align*}
			&\sup _{x \in \mathbb{B}^{n}}\sup_{l \in\mathbb{S}^{n-1}}I(\alpha, q,x , l )\\
			&=\frac{\Gamma\left(\frac{n}{2}\right)\Gamma\left(\frac{q+1}{2}\right) }{\sqrt{\pi}\Gamma\left(\frac{q+n}{2}\right) }{ }_{2} F_{1}\left(n-1-\frac{\left(n-\alpha\right)q}{2}, \frac{n -(n+1-\alpha )q}{2}; \frac{q+n}{2}; 1\right),
		\end{align*}
		where $t_{x}$ is   any unit vector in $\mathbb{R}^{n}$ such that $\langle t_{x}, n_{x}\rangle=0$.
		
		$(4)$   If  $1\leq q<\frac{2n-2}{n-\alpha}$ and $q>\frac{2n}{n-\alpha}$, then
		\begin{align*}
			\sup_{l \in\mathbb{S}^{n-1}}I(\alpha, q,x , l )&= I\left(\alpha,q,x, n_{x}\right)
			=(1+|x|^{2})^{\frac{(n-\alpha)q}{2}-n+1} \frac{\Gamma\left(\frac{n}{2}\right)\Gamma\left(\frac{q+1}{2}\right) }{\sqrt{\pi}\Gamma\left(\frac{q+n}{2}\right) }\times\\
			&{ }_{3} F_{2}\left(\frac{2n-2-( n-\alpha )q}{4},\frac{2n-(n-\alpha)q}{4}, \frac{q+1}{2} ; \frac{1}{2}, \frac{q+n}{2} ; \frac{4|x|^{2}}{(1+|x|^{2})^{2}}\right)
		\end{align*}
		and			\begin{align*}
			\sup _{x \in \mathbb{B}^{n}} \sup_{l \in\mathbb{S}^{n-1}}I(\alpha, q,x , l )=&2^{\frac{(n-\alpha)q}{2}-n+1} \frac{\Gamma\left(\frac{n}{2}\right)\Gamma\left(\frac{q+1}{2}\right) }{\sqrt{\pi}\Gamma\left(\frac{q+n}{2}\right) }\times\\
			&{ }_{3} F_{2}\left(\frac{2n-2-( n-\alpha )q}{4},\frac{2n-(n-\alpha)q}{4}, \frac{q+1}{2} ; \frac{1}{2}, \frac{q+n}{2} ;1\right).
		\end{align*}

	\end{lem}
	\begin{proof}
		We divide the proof of the lemma into four cases according to the range of $q$.
		\begin{case}
			Suppose $q=\frac{2n-2}{n-\alpha}$.
		\end{case}
		By \eqref{eq-2.6} and Corollary \ref{cor-2.2}, we see that for any $l \in \mathbb{S}^{n-1}$,
		$$
		I(\alpha, q,x , l  )\equiv I(\alpha, q,x , l_{0} ) =\int_{\mathbb{S}^{n-1}}|\eta_{1}|^{\frac{2n-2}{n-\alpha}}d\sigma(\eta),
		$$
		where $l_{0}=e_{1}$ and $\eta=(\eta_{1},\ldots,\eta_{n})$. Since
		\begin{align*}
			\int_{\mathbb{S}^{n-1}}|\eta_{1}|^{\frac{2n-2}{n-\alpha}}d\sigma(\eta)
			&=\frac{(n-1)\Gamma(\frac{n}{2})}{2\sqrt{\pi}\Gamma\left(\frac{n+1}{2}\right)}\int_{-1}^{1}(1-t^{2})^{\frac{n-3}{2}}|t|^{\frac{2n-2}{n-\alpha}}dt
			\;\;\text{(by \cite[Equation (2.6)]{KHM2022})}\\
			&=\frac{(n-1)\Gamma(\frac{n}{2})}{\sqrt{\pi}\Gamma\left(\frac{n+1}{2}\right)}\int_{0}^{1}(1-t^{2})^{\frac{n-3}{2}}t^{\frac{2n-2}{n-\alpha}}dt\\
			&=\frac{\Gamma(\frac{n}{2}) \Gamma\left(\frac{3 n-\alpha-2}{2 n-2\alpha  }\right)}{\sqrt{\pi} \Gamma \left(\frac{n^{2}+n(2-\alpha)-2}{2 n-2\alpha }\right) },\qquad\qquad\quad\qquad\text{(by \cite[Equation (4.8)]{KHM2022})}
		\end{align*}
		we obtain that	
		$$
		I(\alpha, q,x , l  )\equiv
		\frac{\Gamma(\frac{n}{2}) \Gamma\left(\frac{3 n-\alpha-2}{2 n-2\alpha }\right)}{\sqrt{\pi} \Gamma \left(\frac{n^{2}+2n -\alpha n -2}{2 n-2\alpha }\right) }.
		$$
		\begin{case}
			Suppose	 $q=\frac{2n}{n-\alpha}$.
		\end{case}
		It follows from \eqref{eq-2.9} and \eqref{eq-2.10}  that
		\begin{align*}
			\mathcal{J} (\alpha,\beta,q,r, |x| )=\int_{-\pi}^{\pi}(1+|x|^{2}-2r|x| \cos\theta)|\cos(\theta-\beta)|^{q}d\theta.
		\end{align*}
		Then, by Lemma \ref{eq-2.3} and the proof of \cite[Case 4.2]{KHM2022}, we   obtain that for any $l \in \mathbb{S}^{n-1}$,
		$$
		I(\alpha, q,x , l )
		\equiv
		\frac{\Gamma(\frac{n}{2})\Gamma(\frac{q+1}{2})}{\sqrt{\pi}\Gamma(\frac{n+q}{2})}(1+|x|^{2})
		=\frac{\Gamma(\frac{n}{2})\Gamma(\frac{3n-\alpha}{2n-2\alpha})}{\sqrt{\pi}\Gamma(\frac{n^{2}+2n -\alpha n }{2n-2\alpha})}(1+|x|^{2}).
		$$
		
		\begin{case}
			Suppose	$\frac{2n-2}{n-\alpha}<q<\frac{2n}{n-\alpha} $.
		\end{case}
		By \eqref{eq-2.6} and Corollary \ref{cor-2.2}, we know that
		$$
		\sup_{l \in\mathbb{S}^{n-1}}I(\alpha, q,x , l  )
		= I(\alpha, q,|x|e_{1} , l_{ \frac{\pi}{2}} )
		=\int_{\mathbb{S}^{n-1}}(1+|x|^{2}-2|x|\eta_{1} )^{\frac{( n-\alpha)q}{2}- n+1}| \eta_{2}|^{q} d \sigma(\eta),
		$$
		where   $\eta=(\eta_{1},\ldots,\eta_{n})$.
		Let $s=n-1-\frac{( n-\alpha)q}{2} $.
		Using a method similar to that in \cite[Case 4.4]{KHM2022}, we obtain
		$$
		\int_{\mathbb{S}^{n-1}}(1+|x|^{2}-2|x|\eta_{1} )^{-s}| \eta_{2}|^{q} d \sigma(\eta)
		=\frac{\Gamma\left(\frac{n}{2}\right)\Gamma\left(\frac{q+1}{2}\right) }{\sqrt{\pi}\Gamma\left(\frac{q+n}{2}\right) }
		{ }_{2} F_{1}\left(s, s- \frac{q+n}{2}+1; \frac{q+n}{2}; |x|^{2}\right).
		$$
		Therefore,
		\begin{align*}
			\sup_{l \in\mathbb{S}^{n-1}}I(\alpha, q,x , l  )=& \frac{\Gamma\left(\frac{n}{2}\right)\Gamma\left(\frac{q+1}{2}\right) }{\sqrt{\pi}\Gamma\left(\frac{q+n}{2}\right) }\times\\
			&{ }_{2} F_{1} \left(n-1-\frac{( n-\alpha)q}{2}, \frac{n-(n+1-\alpha) q}{2} ; \frac{q+n}{2}; |x|^{2}\right).
		\end{align*}
		Further,   \cite[Lemma 1.2]{Ol14} and \cite[Theorem 2.1.2]{an1999} yield to
		\begin{align*}
			\sup _{x \in \mathbb{B}^{n}}&\sup_{l \in\mathbb{S}^{n-1}}I(\alpha, q,x , l  )\\
			&= \frac{\Gamma\left(\frac{n}{2}\right)\Gamma\left(\frac{q+1}{2}\right) }{\sqrt{\pi}\Gamma\left(\frac{q+n}{2}\right) }  { }_{2} F_{1} \left(n-1-\frac{( n-\alpha)q}{2}, \frac{n-(n+1-\alpha) q}{2} ; \frac{q+n}{2}; 1\right)<\infty .
		\end{align*}
		
		\begin{case}
			Suppose $1\leq q<\frac{2n-2}{n-\alpha}$ or $q>\frac{2n}{n-\alpha}$.
		\end{case}
		
		From \eqref{eq-2.6} and Corollary \ref{cor-2.2}, we see that
		$$
		\sup_{l \in\mathbb{S}^{n-1}}I(\alpha, q,x , l  )
		= I(\alpha, q,|x|e_{1} , l_{0} )
		=\int_{\mathbb{S}^{n-1}}(1+|x|^{2}-2|x|\eta_{1} )^{\frac{( n-\alpha)q}{2}- n+1}| \eta_{1}|^{q} d \sigma(\eta),
		$$
		where   $\eta=(\eta_{1},\ldots,\eta_{n})$.
		Let $w=n-1-\frac{( n-\alpha)q}{2} $.
		By a method similar to that used in \cite[Case 4.3]{KHM2022},
		we get that
		\begin{align*}
			&\int_{\mathbb{S}^{n-1}}(1+|x|^{2}-2|x|\eta_{1} )^{-w}| \eta_{1}|^{q} d \sigma(\eta)
			\\=&(1+|x|^{2})^{-w} \frac{\Gamma\big(\frac{n}{2}\big)\Gamma\big(\frac{q+1}{2}\big) }{\sqrt{\pi}\Gamma\big(\frac{q+n}{2}\big) } { }_{3} F_{2}\left(\frac{w}{2}, \frac{w+1}{2},\frac{q+1}{2}; \frac{1}{2}, \frac{q+n}{2} ; \frac{4|x|^{2}}{(1+|x|^{2})^{2}}\right).
		\end{align*}
		Therefore,
		\begin{align*}
			\sup_{l \in\mathbb{S}^{n-1}}&I(\alpha, q,x , l  )
			=(1+|x|^{2})^{\frac{(n-\alpha)q}{2}-n+1} \frac{\Gamma\left(\frac{n}{2}\right)\Gamma\left(\frac{q+1}{2}\right) }{\sqrt{\pi}\Gamma\left(\frac{q+n}{2}\right) }\\
			&\times{ }_{3} F_{2}\left(\frac{2n-2-( n-\alpha )q}{4}, \frac{2n-(  n-\alpha )q}{4}, \frac{q+1}{2} ; \frac{1}{2}, \frac{q+n}{2} ; \frac{4|x|^{2}}{(1+|x|^{2})^{2}}\right).
		\end{align*}
		
		In the following, we will calculate the quantity $\sup _{x \in \mathbb{B}^{n}} \sup_{l \in\mathbb{S}^{n-1}} I(\alpha, q,x , l  )$.
		
		When $q>\frac{2n}{n-\alpha}$, we see from \cite[Lemma 3]{KHM2022} and \cite[Theorem 2.1.2]{an1999} that
		\begin{align*}
			\sup _{x \in \mathbb{B}^{n}} &\sup_{l \in\mathbb{S}^{n-1}} I(\alpha, q,x , l  )
			=2^{\frac{(n-\alpha)q}{2}-n+1} \frac{\Gamma\left(\frac{n}{2}\right)\Gamma\left(\frac{q+1}{2}\right) }{\sqrt{\pi}\Gamma\left(\frac{q+n}{2}\right) }\\
			&\times{ }_{3} F_{2}\left( \frac{2n-2-( n-\alpha )q}{4}, \frac{2n-(  n-\alpha )q}{4},\frac{q+1}{2} ; \frac{1}{2}, \frac{q+n}{2} ; 1\right)<\infty.
		\end{align*}
		
		Next, we consider the case when	  $1\leq q<\frac{2n-2}{n-\alpha}$.
		For any $r\in[0,1]$, it follows from \cite[Equation (2.2.2)]{an1999} that
		$$
		{ }_{3} F_{2}\left(a_{1}, a_{2}, a_{3} ;  \frac{1}{2}, b_{1} ; r\right)
		=\frac{\Gamma\left(b_{1}\right)}{\Gamma\left(a_{3}\right) \Gamma\left(b_{1}-a_{3}\right)} \int_{0}^{1} t^{a_{3}-1}(1-t)^{-a_{3}+b_{1}-1}{ }_{2} F_{1}\left(a_{1}, a_{2} ; \frac{1}{2} ; r t\right) d t,
		$$
		where
		$$a_{1}=\frac{2n-2-( n-\alpha )q}{4},\;a_{2}=\frac{2n-(  n-\alpha )q}{4},\;a_{3}=\frac{q+1}{2}\;\;\text{and}\;\;b_{1}=\frac{q+n}{2} .$$
		It is easy to see from \eqref{eq-1.1}  that the mapping ${ }_{2} F_{1}\left(a_{1}, a_{2}; \frac{1}{2}; r\right)$  is increasing in $[0,1)$.
		Thus, ${ }_{3} F_{2}\left(a_{1}, a_{2}, a_{3} ;  \frac{1}{2}, b_{1} ; r\right)$ is increasing in $[0,1)$.
		This, together with  \cite[Theorem 2.1.2]{an1999}, guarantees that
		\begin{align*}
			\sup _{x \in \mathbb{B}^{n}}&\sup_{l \in\mathbb{S}^{n-1}} I(\alpha, q,x , l  )
			=2^{\frac{(n-\alpha)q}{2}-n+1} \frac{\Gamma\left(\frac{n}{2}\right)\Gamma\left(\frac{q+1}{2}\right) }{\sqrt{\pi}\Gamma\left(\frac{q+n}{2}\right) }\\
			&\times{ }_{3} F_{2}\left(\frac{2n-2-( n-\alpha )q}{4},\frac{2n-(  n-\alpha )q}{4}, \frac{q+1}{2} ; \frac{1}{2}, \frac{q+n}{2} ; 1\right)
			<\infty.
		\end{align*}
		The proof of the lemma is complete.
	\end{proof}

	\subsection{Proof of Theorem \ref{thm-1.1}}
	For any  $x \in \mathbb{B}^{n}$,
	by \eqref{eq-2.4} and  Lemma \ref{lem-2.2},
	we have
	\begin{align*}
		\mathbf{C}_{\alpha, q}(x ; l)
		\leq\frac{(n-\alpha)C_{n,\alpha}}{\left(1-|x|^{2}\right)^{\frac{n(q-1)+1}{q}}} \big(I(\alpha,q,x,l) +  J(\alpha,q,x )\big)^{\frac{1}{q}},
	\end{align*}
	where  $ J(\alpha,q,x )$ is the mapping from
	\eqref{eq-2.7} and
	the quantity  $\sup_{l \in\mathbb{S}^{n-1}}I(\alpha,q,x,l)$ is obtained in Lemma \ref{lem-2.4}.
	Combining the above inequality with \eqref{eq-2.1} and \eqref{eq-2.2},
	we see that \eqref{eq-1.5}  holds true.
	Further, by \eqref{eq-1.4}, Lemmas \ref{lem-2.2}, \ref{lem-2.4}   and \cite[Theorem 1]{KHM2022},
	we see that \eqref{eq-1.5} is sharp when $x=0$ or $\alpha=2-n$.
	\qed
	
	\section{Proof of Theorem \ref{thm-1.2}}	
	The aim of this section is to prove   Theorem \ref{thm-1.2}.
	In fact, it can be derived directly from   \eqref{eq-2.1}, Lemmas \ref{lem-3.1} and \ref{lem-3.4}.
	
	\begin{lem}\label{lem-3.1}
		For any $\alpha<1$, $ x \in \mathbb{B}^{n}$  and  $l \in \mathbb{S}^{n-1}$, we have
		$$
		\mathbf{C}_{\alpha, \infty}(x ; l)=\sup _{\zeta \in \mathbb{S}^{n-1}}\left|\left\langle\nabla {P}_{\alpha}(x, \zeta), l\right\rangle\right| \text { and } \mathbf{C}_{\alpha, \infty}(0) \equiv \mathbf{C}_{\alpha, \infty}(0 ; l) \equiv (n-\alpha)C_{n,\alpha}.
		$$
	\end{lem}
	\begin{proof}
		The proof of this result is similar to that of \cite[Lemma 4.1]{chen2023B}, where $P_{\alpha}$ is used
		instead of $\mathcal{P}_{\mathbb{B}^{n}}$,  and we omit it.
	\end{proof}
	
	\medskip
	
	For any  $\beta \in[0, \pi]$ and $\rho \in[0,1) $, we let
	\begin{align}\label{eq-3.1}
		C_{\alpha, \infty}&  (\rho e_{1} ; l_{\beta} )\notag=\max _{\zeta \in \mathbb{S}^{n-1}} \frac{1}{\left|\rho e_{1}-\zeta\right|^{n+2-\alpha}}\big| (n-\alpha)(1-\rho^{2}) \sin \beta \cdot \zeta_{2} \notag\\
		&- \big((2 -3 \alpha + n) \rho   + (2 - \alpha - n) \rho^3 + (\alpha -
		n + (-4 +3\alpha + n) \rho^2) \zeta_{1} \big) \cos \beta\big|,
	\end{align}
	where $l_{\beta}=\cos \beta \cdot e_{1}+\sin \beta \cdot e_{2}$  and $\zeta=\left(\zeta_{1}, \zeta_{2}, \ldots, \zeta_{n}\right) \in \mathbb{S}^{n-1}$. Then, we obtain the following result.
	\begin{lem}\label{lem-3.2}
		For any $\alpha<1$, $\beta \in[0, \pi]$ and $ \rho \in[0,1)$, we have
		\begin{equation}\label{eq-3.2}
			\mathbf{C}_{\alpha, \infty}\left(\rho e_{1} ; l_{\beta}\right)=C_{n,\alpha}\left(1-\rho^{2}\right)^{-\alpha} C_{\alpha, \infty}\left(\rho e_{1} ; l_{\beta}\right).
		\end{equation}
	\end{lem}
	\begin{proof}
		For any  $x \in \mathbb{B}^{n}$ and $l \in \mathbb{S}^{n-1}$,  it follows from \eqref{eq-1.2}  that
		\begin{align}\label{eq-3.3}
			\sup _{\zeta \in \mathbb{S}^{n-1}}\left|\left\langle\nabla P_{\alpha}(x, \zeta), l\right\rangle\right| \notag
			= & C_{n,\alpha}\left(1-|x|^{2}\right)^{-\alpha} \\
			&\times\max _{\zeta \in \mathbb{S}^{n-1}} \frac{\left|\left.\langle 2(1-\alpha)x| x-\left.\zeta\right|^{2}+(n-\alpha)\left(1-|x|^{2}\right)(x-\zeta), l\right\rangle\right|}{|x-\zeta|^{n+2-\alpha}}.
		\end{align}
		Let $x=\rho e_{1}$ and  $l=l_{\beta}$.
		By   straightforward calculations, we get
		\begin{align*}
			& \max _{\zeta \in \mathbb{S}^{n-1}} \frac{\left|\left.\langle 2(1-\alpha)x| x-\left.\zeta\right|^{2}+(n-\alpha)\left(1-|x|^{2}\right)(x-\zeta), l\right\rangle\right|}{|x-\zeta|^{n+2-\alpha}}\\
			= & \max _{\zeta \in \mathbb{S}^{n-1}} \frac{1}{\left|\rho e_{1}-\zeta\right|^{n+2-\alpha}}\big| (n-\alpha)(1-\rho^{2}) \sin \beta \cdot \zeta_{2} \notag\\
			&- \big((2 -3 \alpha + n) \rho   + (2 - \alpha - n) \rho^3 + (\alpha -
			n + (-4 +3\alpha + n) \rho^2) \zeta_{1} \big) \cos \beta\big|.
		\end{align*}
		This, together with Lemma \ref{lem-3.1}, \eqref{eq-3.1} and \eqref{eq-3.3}, implies that \eqref{eq-3.2} holds true.
	\end{proof}

	\begin{lem}\label{lem-3.3}
		For any $ \alpha<1$, $x \in \mathbb{B}^{n}$, $l \in \mathbb{S}^{n-1}$  and unitary transformation  $A$  in $\mathbb{R}^{n}$, we have
		$$
		\mathbf{C}_{\alpha, \infty}(x ; l)=\mathbf{C}_{\alpha, \infty}(A x ; A l).
		$$
	\end{lem}

	\begin{proof}
		The proof of this result is similar to that of \cite[Lemma 4.1]{chen2023B}, where $P_{\alpha}$, Lemma \ref{lem-3.1} and \eqref{eq-3.3} is used
		instead of $\mathcal{P}_{\mathbb{B}^{n}}$, \cite[Lemma 4.1]{chen2023B} and \cite[Equation (4.7)]{chen2023B}, respectively.
		Hence, we omit the proof of the lemma.
	\end{proof}

	\begin{lem}\label{lem-3.4}
		Let  $l\in\mathbb{S}^{n-1}$ and $x \in \mathbb{B}^{n}\backslash\{0\} $.
		
		$(1)$  If $2-n\leq\alpha<1$, then
		$$
		\mathbf{C}_{\alpha, \infty}(x ; l) \leq \mathbf{C}_{\alpha, \infty}\left(x ; \pm n_{x}\right)=\mathbf{C}_{\alpha, \infty}(x)=C_{n,\alpha} \frac{ n-\alpha+(n+\alpha-2)|x| }{(1-|x|)^{n}}.
		$$
		
		$(2)$	If $\alpha<2-n$, then
		$$
		C_{n,\alpha} \frac{ n-\alpha+(n+\alpha-2)|x| }{(1-|x|)^{n}}\leq \mathbf{C}_{\alpha, \infty}(x)\leq C_{n,\alpha} \frac{ n-\alpha-(n+\alpha-2)|x| }{(1-|x|)^{ n}}.
		$$
	\end{lem}
	\begin{proof}
		For any  $x \in \mathbb{B}^{n} \backslash\{0\}  $  and  $l \in \mathbb{S}^{n-1} $, let  $\beta=\arccos \left\langle n_{x}, l\right\rangle \in[0, \pi]$
		and $A$ be a unity transformation in  $\mathbb{R}^{n}$  such that $A x=\rho e_{1}$  and $A l=l_{\beta}$,
		where $\rho=|x| \in(0,1)$.
		Using \eqref{eq-2.2}, Lemmas \ref{lem-3.2} and \ref{lem-3.3}, we get that
		\begin{align}\label{eq-3.4}
			\mathbf{C}_{\alpha, \infty}(x) & =\sup _{l \in \mathbb{S}^{n-1}} \mathbf{C}_{\alpha, \infty}(x ; l)=\sup _{l \in \mathbb{S}^{n-1}} \mathbf{C}_{\alpha, \infty}(A x ; A l)=\sup _{\beta \in[0, \pi]} \mathbf{C}_{\alpha, \infty}\left(\rho e_{1} ; l_{\beta}\right) \notag\\
			& =C_{n,\alpha}\left(1-\rho^{2}\right)^{-\alpha} \sup _{\beta \in[0, \pi]}C_{\alpha, \infty}\left(\rho e_{1} ; l_{\beta}\right).
		\end{align}
		Hence, to prove the lemma, it suffices to estimate the quantity  $\sup _{\beta \in[0, \pi]} C_{\alpha, \infty}\left(\rho e_{1} ; l_{\beta}\right)$.
		\begin{claim}\label{claim-3.1}
			$(1)$ If  $2-n\leq\alpha<1$, then
			$$
			\sup _{\beta \in[0, \pi]} C_{\alpha, \infty}\left(\rho e_{1} ; l_{\beta}\right)=C_{\alpha, \infty}\left(\rho e_{1} ; \pm e_{1}\right)=\frac{ n-\alpha+(n+\alpha-2)\rho }{(1-\rho)^{ n-\alpha}};$$
			$(2)$ If $\alpha<2-n$, then
			$$
			\frac{ n-\alpha+(n+\alpha-2)\rho }{(1-\rho)^{n-\alpha}}\leq \sup _{\beta \in[0, \pi]} C_{\alpha, \infty}\left(\rho e_{1} ; l_{\beta}\right)\leq \frac{  n-\alpha-(n+\alpha-2)\rho }{(1-\rho)^{ n-\alpha}}.
			$$
		\end{claim}
		By \eqref{eq-3.1} and spherical coordinate transformation (cf. \cite[Equation (2.2)]{chen2018}), we find
		\begin{align*}
			C_{\alpha, \infty}&\left(\rho e_{1} ; l_{\beta}\right)
			=\max _{\theta \in[0, \pi]} \max _{\gamma \in[0, \pi]}   \frac{1}{\left|1+\rho^{2}-2\rho \cos\theta\right|^{\frac{n-\alpha}{2}+1}} \big| (n-\alpha)(1-\rho^{2}) \sin \beta  \sin\theta \cos\gamma\\
			& -\big((2 -3 \alpha + n) \rho+ (2 - \alpha - n) \rho^3+ (\alpha -
			n + (-4 +3\alpha + n) \rho^2) \cos\theta\big) \cos \beta\big| .
		\end{align*}
		Obviously, the maximum in  $\gamma$  is attained either at  $\gamma=0$  or  $\gamma=\pi$.
		Hence,
		\begin{align}\label{eq-3.5}
			&C_{\alpha, \infty} (\rho e_{1} ; l_{\beta} )=\max_{\theta \in[0,2 \pi]}\frac{1}{ |1+\rho^{2}-2\rho \cos\theta |^{\frac{n-\alpha}{2}+1}}\big|  (n-\alpha)(1-\rho^{2}) \sin \beta  \sin\theta   \notag\\
			& \pm \big((2 -3 \alpha + n) \rho+ (2 - \alpha - n) \rho^3 + ( \alpha -
			n + (-4 +3\alpha + n) \rho^2) \cos\theta\big) \cos \beta\big| .
		\end{align}
		This yields to
		\begin{align}\label{eq-3.6}
			&\sup _{\beta \in[0, \pi]}C_{\alpha, \infty}\left(\rho e_{1} ; l_{\beta}\right)  \geq C_{\alpha, \infty}\left(\rho e_{1} ;  e_{1}\right) \\
			&=\max _{\theta \in[0,2 \pi]} \frac{\left|(2 -3 \alpha + n) \rho+ (2 - \alpha - n) \rho^3 + ( \alpha -
				n + (-4 +3\alpha + n) \rho^2) \cos\theta\right|}{\left(1+\rho^{2}-2 \rho \cos \theta\right)^{\frac{n-\alpha}{2}+1}}\notag \\
			& \geq \frac{\left|(2 -3 \alpha + n) \rho+ (2 - \alpha - n) \rho^3 + ( \alpha -
				n + (-4 +3\alpha + n) \rho^2) \right|}{(1-\rho)^{n-\alpha+2}}\notag\\\nonumber
			&=\frac{ n-\alpha+(n+\alpha-2)\rho }{(1-\rho)^{ n-\alpha}}.
		\end{align}
		On the other hand, by \eqref{eq-3.5} and Cauchy-Schwarz inequality, we have
		\begin{align}\label{eq-03.7}
			&C_{\alpha, \infty}(\rho e_{1} ; l_{\beta} )
			\leq \frac{1}{(1-\rho)^{n-\alpha}} \max _{\theta \in[0,2 \pi]}
			\left(\frac{ ( n-\alpha)^{2}(1-\rho^{2})^{2} \sin^{2} \theta }{ (1+\rho^{2}-2 \rho \cos \theta )^{2}}\right. \\
			&\;\;+\left.\frac{\big((2 -3 \alpha + n) \rho+ (2 - \alpha - n) \rho^3 + ( \alpha -
				n + (-4 +3\alpha + n) \rho^2) \cos\theta \big)^{2}}{ (1+\rho^{2}-2 \rho \cos\theta  )^{2}}
			\right)^{\frac{1}{2}}\notag\\
			&=\frac{1}{(1-\rho)^{n-\alpha}} \max _{\theta \in[0,2 \pi]}
			K^{\frac{1}{2}}_{\alpha}(\cos\theta),
			\notag\\\nonumber
		\end{align}
		where
		\begin{align*}
			K_{\alpha}(t)=&\frac{1}{1+\rho^{2}-2 \rho t}\left((- \alpha + n)^2 +
			2 (2 -3 \alpha (2 -\alpha) - (-2 + n) n) \rho^2 \right.\\
			&\left. + (-2 + \alpha + n)^2 \rho^4- 4 (1 -\alpha) \rho (- \alpha + n + (2 - \alpha - n) \rho^2) t\right)
		\end{align*}
		for any $t\in[-1,1]$.
		Basic calculations show that
		$$
		\frac{d}{d t}K_{\alpha }(t)= \frac{2(n-\alpha)(n+\alpha-2)  \rho\left(-1+\rho^{2}\right)^{2}}{\left(1+\rho^{2}-2 \rho t\right)^{2}}.
		$$
		Therefore, when   $2-n\leq\alpha<1$ and  $\rho\in(0,1)$,  $K_{\alpha}(t)$ is   increasing  in $[-1,1]$ and
		\begin{equation}\label{eq-3.8}
			\max_{t\in[-1,1]} K_{\alpha}(t)=K_{\alpha}(1)=(n-\alpha+(n+\alpha-2)\rho)^{2};
		\end{equation}
		when   $\alpha<2-n$ and $\rho\in(0,1)$, $K_{\alpha}(t)$ is  decreasing   in $[-1,1]$ and
		\begin{equation}\label{eq-3.9}
			\max_{t\in[-1,1]} K_{\alpha}(t)=K_{\alpha}(-1)=(\alpha-n+(n+\alpha-2)\rho)^{2}.
		\end{equation}
		
		Now, it follows from \eqref{eq-3.6}-\eqref{eq-3.8} that
		$$
		\sup _{\beta \in[0, \pi]} C_{\alpha, \infty}\left(\rho e_{1} ; l_{\beta}\right)=C_{\alpha, \infty}\left(\rho e_{1} ; \pm e_{1}\right)=\frac{ n-\alpha+(n+\alpha-2)\rho }{(1-\rho)^{ n-\alpha}}
		$$
		for $2-n\leq\alpha<1$;
		Similarly,    \eqref{eq-3.6}, \eqref{eq-03.7} and \eqref{eq-3.9} ensure that
		$$
		\frac{ n-\alpha+(n+\alpha-2)\rho }{(1-\rho)^{ n-\alpha}}\leq \sup _{\beta \in[0, \pi]} C_{\alpha, \infty}\left(\rho e_{1} ; l_{\beta}\right)\leq \frac{  n-\alpha-(n+\alpha-2)\rho }{(1-\rho)^{ n-\alpha}}
		$$
		for  $\alpha<2-n$.
		Thus, Claim 3.1 is true.
		
		By \eqref{eq-3.4} and Claim \ref{claim-3.1},  we see that Lemma \ref{lem-3.4} is true.
	\end{proof}

	
	\section{Proof of Theorem \ref{thm-1.3}}
	
	The aim of this section is to prove Theorem \ref{thm-1.3}.
	\subsection{Proof of Theorem \ref{thm-1.3}}
	
	For any $r\in(0,1)$ and $x^{\prime}$, $x^{\prime \prime} \in   \mathbb{B}^{n}\left( r\right)$, we have that
	\begin{align}\label{eq-4.1}
		\left|u\left(x^{\prime}\right)-u\left(x^{\prime \prime}\right)\right|
		\geq & \left|\int_{\left[x^{\prime}, x^{\prime \prime}\right]} Du(0) d x\right|-\int_{\left[x^{\prime}, x^{\prime \prime}\right]}\|Du(x)-Du(0)\| d x \\\nonumber
		\geq &    \int_{ [x^{\prime}, x^{\prime \prime} ]}   \big( l (Du(0))  -\|D u(0)-Du(x)\| \big) d x,
	\end{align}
	where $\left[x^{\prime}, x^{\prime \prime}\right]$    denotes the segment from  $x^{\prime}$  to  $x^{\prime \prime}$  with the endpoints  $x^{\prime}$  and  $x^{\prime \prime}$.
	
	In order to prove our theorem,
	we still  need to estimate the quantities $l\big(Du(0)\big)$ and $\|D u(0)-Du(x)\|$.
	
	\medskip

	\noindent \textbf{$(\mathfrak{A})$ Estimate on   $l\big(Du(0)\big)$.}
	Since   $\|\phi\|_{L^{\infty}(\mathbb{S}^{n-1},\mathbb{R}^{n})}\leq M$ for some  constant $M>0$,
	by \eqref{eq-1.4}, we get that
	$$
	\|D u(0)\| \leq  M (n-\alpha)\frac{C_{n,\alpha}\Gamma\left(\frac{n}{2}\right)}{\sqrt{\pi} \Gamma\left(\frac{n+1}{2}\right)} =:N_{*}.
	$$
	Then   \cite[Lemma 4]{LXY1992}  and the assumption ``$J_{u}(0)=1$"  guarantee that
	\begin{equation}\label{eq-4.2}
		l\big(Du(0)\big) \geq \frac{J_{u}(0)}{\|Du(0)\|^{n-1}} \geq \frac{1}{ N_{*}^{n-1}}.
	\end{equation}

	\noindent \textbf{$(\mathfrak{B})$ Estimate on   $\|Du(0)-Du(x)\|$.}
	For any   $x \in \mathbb{B}^{n}  $, it follows from \eqref{eq-1.2} that
	$$
	\nabla_{x}  P _{\alpha}(x, \zeta) = -C_{n,\alpha} \frac{2(1 - \alpha) x |x - \zeta|^{2} + (n - \alpha)(1 - |x|^{2})(x - \zeta)}{(1 - |x|^{2})^{ \alpha}|x - \zeta|^{n + 2 - \alpha}}.
	$$
	Then for any $\eta\in\mathbb{S}^{n-1}$,
	\begin{align*}
		\big|\big(&Du (0)-Du(x)\big)\eta\big|\\
		=&\left|\int_{\mathbb{S}^{n-1}} \big \langle \nabla_{x}  P_{\alpha}(0, \zeta)-\nabla_{x}  P _{\alpha}(x, \zeta),\eta \big\rangle \phi(\zeta)d\sigma(\zeta)\right| \\
		=& C_{n,\alpha}
		\bigg|(n-\alpha)\int_{\mathbb{S}^{n-1}}\langle\zeta,\eta\rangle\phi(\zeta) d\sigma(\zeta)\\
		&+
		\int_{\mathbb{S}^{n-1}} \frac{ 2(1-\alpha)\langle x,\eta\rangle|x-\zeta|^{2}
			+(n-\alpha) (1-|x|^{2})\langle x-\zeta ,\eta\rangle }{(1-|x|^{2})^{ \alpha}|x-\zeta|^{n+2-\alpha}}\phi(\zeta) d\sigma(\zeta)
		\bigg| .
	\end{align*}
	This, together with the assumption $\|\phi \|_{L^{\infty}(\mathbb{S}^{n-1},\mathbb{R}^{n})}\leq M$,
	yields to
	\begin{align}\label{eq-4.3}
		& \|Du(0)-Du(x)\|\leq
		M (n-\alpha) C_{n,\alpha}\int_{\mathbb{S}^{n-1}} \left|1-\frac{(1-|x|^{2})^{1-\alpha}}{|x-\zeta|^{n+2-\alpha}} \right|d\sigma(\zeta)\\\nonumber
		&+ MC_{n,\alpha}|x|(1-|x|^{2})^{-\alpha}\left(\int_{\mathbb{S}^{n-1}}\frac{2(1-\alpha)}{|x-\zeta|^{n-\alpha}}d\sigma(\zeta)
		+\int_{\mathbb{S}^{n-1}}\frac{(n-\alpha)(1-|x|^{2})}{|x-\zeta|^{n+2-\alpha}}d\sigma(\zeta)\right).
	\end{align}
	
	Now, we estimate the right-hand side of the above inequality.
	By  calculations, we get
	\begin{align*}
		|&x-\zeta|^{n+2-\alpha}-(1-|x|^{2})^{1-\alpha} \\
		&= (1-2\langle x,\zeta\rangle+|x|^{2})^{\frac{n+2-\alpha}{2}}-(1-|x|^{2})^{1-\alpha} \\
		&= (1+|x|^{2})^{\frac{n+2-\alpha}{2}}\sum_{k=0}^{\infty}\frac{(\frac{\alpha-n-2}{2})_{k}}{k!}\left(\frac{2\langle x,\zeta\rangle}{1+|x|^{2}}\right)^{k}-\sum_{k=0}^{\infty}\frac{(\alpha-1)_{k}}{k!}|x|^{2k} \\
		&= \sum_{k=1}^{\infty}\binom{\frac{n+2-\alpha}{2}}{k}|x|^{2k}
		+(1+|x|^{2})^{\frac{n+2-\alpha}{2}}\sum_{k=1}^{\infty}\frac{(\frac{\alpha-n-2}{2})_{k}}{k!}\left(\frac{2\langle x,\zeta\rangle}{1+|x|^{2}}\right)^{k}
		-\sum_{k=1}^{\infty}\frac{(\alpha-1)_{k}}{k!}|x|^{2k}.
	\end{align*}
	Hence,
	\begin{align*}
		&\big| |x-\zeta|^{n+2-\alpha}-(1-|x|^{2})^{1-\alpha} \big|\\	
		& \leq\left|\sum_{k=1}^{\infty}\binom{\frac{n+2-\alpha}{2}}{k}|x|^{2k}\right|
		+(1+|x|^{2})^{\frac{n+2-\alpha}{2}}\left|\sum_{k=1}^{\infty}\frac{(\frac{\alpha-n-2}{2})_{k}}{k!}\frac{2^{k}|x|^{k}}{\left(1+|x|^{2}\right)^{k}}\right|
		+\left|\sum_{k=1}^{\infty}\frac{(\alpha-1)_{k}}{k!}|x|^{2k}\right|\\
		& =(1 +|x|^{2})^{\frac{n+2-\alpha}{2}} +(1+|x|^{2})^{ \frac{n+2-\alpha}{2}}\left(1-\left(1-\frac{2|x|}{1+|x|^{2}}\right)^{\frac{n+2-\alpha }{2} } \right) -(1-|x|^{2})^{ 1-\alpha} \\
		&= 2(1 +|x|^{2})^{\frac{n+2-\alpha}{2}}
		-(1-|x| )^{n+2-\alpha  }   -(1-|x|^{2})^{ 1-\alpha},
	\end{align*}
	which implies that
	\begin{equation}\label{eq-4.4}
		\left|1-\frac{(1-|x|^{2})^{1-\alpha}}{|x-\zeta|^{n+2-\alpha}} \right|\leq \frac{ |x|g(|x|) }{|x - \zeta|^{n + 2 - \alpha}},
	\end{equation}
	where $g(|x|)$  is the mapping from  \eqref{eq-1.7}.
	
	On the other hand, by \cite[Lemma 2.1]{Liu04} and \cite[Theorem 2.2.5]{an1999}, we know that for any $\lambda\in \mathbb{R}$,
	\begin{align*}
		\int_{\mathbb{S}^{n-1} } \frac{ d\sigma(\zeta)}{ |x-\zeta|^{2\lambda}}
		&={ }_{2} F_{1}\left(\lambda,\lambda-\frac{n}{2}+1; \frac{n}{2} ; |x|^{2}\right)\\
		&= (1-|x|^{2})^{n-2\lambda-1}{ }_{2} F_{1}\left(\frac{n}{2}-\lambda, n-\lambda-1; \frac{n}{2} ; |x|^{2}\right).
	\end{align*}
	This, together with  \eqref{eq-4.3} and \eqref{eq-4.4},   provides that
	\begin{align*}
		\|Du(0)-Du(x)\|
		&\leq
		2M(1-\alpha) C_{n,\alpha}|x|(1-|x|^{2})^{-1}{}_{2}F_{1}\left(\frac{\alpha}{2} , \frac{n+\alpha}{2}-1; \frac{n}{2} ;|x|^{2}\right)\\
		&\hspace{-2cm}+M(n-\alpha) C_{n,\alpha}|x|(1-|x|^{2})^{-2}
		{}_{2}F_{1}\left(\frac{\alpha}{2}-1 , \frac{n+\alpha}{2}-2; \frac{n}{2} ;|x|^{2}\right)\\
		&\hspace{-2cm}+M(n-\alpha) C_{n,\alpha} |x|(1-|x|^{2})^{\alpha-3}{}_{2}F_{1}\left(\frac{\alpha}{2}-1 , \frac{n+\alpha}{2}-2; \frac{n}{2} ;|x|^{2}\right)g(|x|).
	\end{align*}
	By \cite[Theorem 2.1.2]{an1999}, we see that
	both ${}_{2}F_{1}\left(\frac{\alpha}{2}, \frac{n-2+\alpha}{2}; \frac{n}{2};t\right)$ and ${}_{2}F_{1}\left(\frac{\alpha}{2}-1, \frac{n+\alpha}{2}-2; \frac{n}{2};t\right)$ are bounded on $[0, 1]$.
	Obviously, $g(t)$ is also bounded in $[0,1)$.
	Therefore,  for any $x \in \mathbb{B}^{n}\left(0, r\right)$,
	\begin{equation}\label{eq-4.5}
		\|Du(0)-Du(x)\|\leq M|x| G(r)<M r G(r),
	\end{equation}
	where $G(r)$ is given by \eqref{eq-1.6}.

	Now, we  show that there exists a constant  $r_{0}\in(0,1)$ such that $u$ is injective in $\mathbb{B}^{n}( r_{0})$.
	For any $r\in(0,1)$ and distinct points $x^{\prime}$, $x^{\prime \prime} \in   \mathbb{B}^{n}\left( r\right)$,
	we	 conclude from \eqref{eq-4.1}, \eqref{eq-4.2} and 	\eqref{eq-4.5} that
	\begin{align*}
		\left|u\left(x^{\prime}\right)-u\left(x^{\prime \prime}\right)\right|
		> |x^{\prime}-x^{\prime \prime}|\psi (r),
	\end{align*}
	where $$\psi (r)= \frac{1}{ N_{*}^{n-1}}-M rG(r). $$
	It is easy to check that
	$$
	\lim _{r \rightarrow 0^{+}} \psi(r)= \frac{1}{ N_{*}^{n-1}}>0\;\;\text{and}\;\;\lim _{r \rightarrow 1^{-}} \psi(r)=-\infty.
	$$
	Let $r_{0}$ be the smallest
	positive root of the equation    $ \psi(r )=0$.
	Then $u$ is injective in $\mathbb{B}^{n}( r_{0})$.
	
	Finally, we prove that the image $u(\mathbb{B}^{n}( r_{0}))$ contains a ball $\mathbb{B}^{n}( R_{0})$.
	For any  $\varsigma\in \{x\in\mathbb{R}^{n}:|x|=r_{0}\} $, by \eqref{eq-4.1}, \eqref{eq-4.2} and 	\eqref{eq-4.5},
	we obtain that
	\begin{align*}
		|u(\varsigma)-u(0)| \geq & \left|\int_{[0, \varsigma]} Du(0) d x\right|-\int_{[0, \varsigma]}\|Du(x)-Du(0)\| d x \\
		\geq & \int_{[0, \varsigma]} \left(\frac{1}{ N_{*}^{n-1}}-M|x|G(r_{0})\right)d x
		= R_{0},
	\end{align*}
	where $	R_{0}=\frac{M}{2} r^{2}_{0}G(r_{0}) $.	
	Since $u(0)=0$, we see that $\mathbb{B}^{n}( R_{0})\subset u(\mathbb{B}^{n}( r_{0}))$,
	and
	the proof of the theorem is complete.
	\qed

	\vspace*{3mm}
	\subsection*{ Acknowledgement.}
	The first author is
	supported by SERB-SRG (SRG/2023/001938).
	The second author is supported by Changsha Municipal Natural Science Foundation (No. kq2502155)
	and the construct program of the key discipline in Hunan Province.
	The   third author is supported by the Institute Post Doctoral
	Fellowship of IIT Madras, India.  
	The  fourth author is supported by NSF of China
	(No.12371071 and No.12571081).
	Some of the research for this paper was carried out in April 2025, when the first and third authors visited the School of Mathematics and Statistics at Hunan Normal University. They express their gratitude to the members of the department for their warm hospitality during that visit.

	\vspace*{5mm}
	\noindent{\bf Data availability.}
	Our manuscript does not have associated data.

	\vspace*{5mm}
	\noindent{\bf Conflict of interest.}
	The authors state that there is no conflict of interest.

\end{document}